\documentclass[12pt]{amsart}
\usepackage[headings]{fullpage}
\usepackage{amssymb,epic,eepic,epsfig,amsbsy,amsmath,amscd,color}
\usepackage[all]{xy}
\usepackage{graphicx,subcaption}
\usepackage{texdraw}
\usepackage{url}
\usepackage{bbm}
\usepackage{mathrsfs}
\usepackage{standalone}
\usepackage{tikz}
\usetikzlibrary{cd,positioning,knots,patterns,hobby}
\usepackage{accents}
\usepackage[normalem]{ulem}
\usepackage{comment}
\usepackage{calc}
\usepackage{soul,xcolor}

\setstcolor{red}

\def\biblbl{\bibitem}

\theoremstyle{plain}
\newtheorem{theorem}{Theorem}[section]
\newtheorem{thm}{Theorem}
\newtheorem{lemma}[theorem]{Lemma}
\newtheorem{corollary}[theorem]{Corollary}
\newtheorem{proposition}[theorem]{Proposition}
\newtheorem{conjecture}{Conjecture}
\newtheorem{question}{Question}

\newtheorem{definition}{Definition}

\theoremstyle{definition}

\newtheorem{remark}[theorem]{Remark}
\newtheorem{example}[theorem]{Example}

\newcommand{\bcon}{\begin{conjecture}}
\newcommand{\econ}{\end{conjecture}}
\newcommand{\bcor}{\begin{corollary}}
\newcommand{\ecor}{\end{corollary}}
\newcommand{\bdf}{\begin{definition}}
\newcommand{\edf}{\end{definition}}
\newcommand{\benu}{\begin{enumerate}}
\newcommand{\eenu}{\end{enumerate}}
\newcommand{\beq}{\begin{equation}}
\newcommand{\eeq}{\end{equation}}
\newcommand{\bexa}{\begin{example}}
\newcommand{\eexa}{\end{example}}
\newcommand{\bexe}{\begin{exercise}}
\newcommand{\eexe}{\end{exercise}}
\newcommand{\bfac}{\begin{fact}}
\newcommand{\efac}{\end{fact}}
\newcommand{\bite}{\begin{itemize}}
\newcommand{\eite}{\end{itemize}}
\newcommand{\blem}{\begin{lemma}}
\newcommand{\elem}{\end{lemma}}
\newcommand{\bmat}{\begin{pmatrix}}
\newcommand{\emat}{\end{pmatrix}}
\newcommand{\bprb}{\begin{problem}}
\newcommand{\eprb}{\end{problem}}
\newcommand{\bpro}{\begin{proposition}}
\newcommand{\epro}{\end{proposition}}

\newcommand{\bque}{\begin{question}}
\newcommand{\eque}{\end{question}}
\newcommand{\brem}{\begin{remark}}
\newcommand{\erem}{\end{remark}}
\newcommand{\bthm}{\begin{theorem}}
\newcommand{\ethm}{\end{theorem}}

\newcommand\FIGc[3]{\begin{figure}[htpb]
    \includegraphics[height=#3]{draws/#1.eps}
    \caption{#2}
    \label{fig:#1}
    \end{figure}}

\newcommand\incl[2]{{\includegraphics[height=#1]{draws/#2.eps}}}

\newcommand{\no}[1]{}

\def\BC{\mathbb C}
\def\BN{\mathbb N}
\def\BZ{\mathbb Z}
\def\BR{\mathbb R}

\def\cP{\mathcal P}

\def\la{\langle}
\def\ra{\rangle}

\def\al{\alpha}

\def\be { \begin{equation} }
\def\ee { \end{equation} }

\def\bt{{\mathbf {t}}}

\def\cS{\mathscr S}

\def\cE{\mathcal E}

\def\bk{\mathbf k}

\def\bn{\mathbf n}

\def\cP{\mathcal P}

\def\embed{\hookrightarrow}

\def\beq{\overset \bullet =}

\def\SM{\cS(M)}
\def\SS{\cS(\Sigma)}

\def\pM{\partial M}
\def\cN{\mathcal N}

\def\fr{\operatorname{fr}}

\def\cSp{\cS^+}

\begin{document}

\title[Faithfulness of Geometric action of skein algebra]{Faithfullness of geometric action of skein algebras}

\author[Thang  T. Q. L\^e]{Thang  T. Q. L\^e}
\address{School of Mathematics, 686 Cherry Street,
 Georgia Tech, Atlanta, GA 30332, USA}
\email{letu@math.gatech.edu}


\thanks{
2010 {\em Mathematics Classification:} Primary 57N10. Secondary 57M25.}

\begin{abstract}
We show that the  action of the Kauffman bracket skein algebra of a surface $\Sigma$  on the skein module of the handlebody bounded by $\Sigma$ is faithful if and only if the quantum parameter is not a root of 1.

\end{abstract}

\maketitle

\def\cSp{\cS^+}

\def\cross{  \raisebox{-8pt}{\incl{.8 cm}{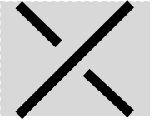}} }
\def\resoP{  \raisebox{-8pt}{\incl{.8 cm}{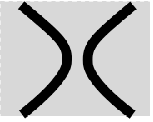}} }
\def\resoN{  \raisebox{-8pt}{\incl{.8 cm}{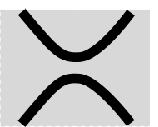}} }
\def\trivloop{  \raisebox{-8pt}{\incl{.8 cm}{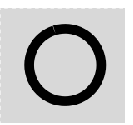}} }
\def\emptyr{\raisebox{-8pt}{\incl{.8 cm}{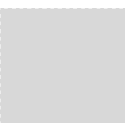}}}
\def\cR{{\mathcal R}}
\section{Introduction} 

\subsection{Kauffman bracket skein modules and algebras}   Let $\cR$ be a commutative domain with unit, equipped with a distinguished invertible element $q$. Suppose $M$ is an oriented 3-manifold with possibly non-empty boundary $\pM$. The skein module $\SM$, introduced by Przytycki~\cite{Przy} and Turaev~\cite{Turaev,Turaev2}, is  the $\cR$-module generated by isotopy classes of framed unoriented links in $M$  modulo the Kauffman relations \cite{Kauffman}
\begin{align}
\label{eq.skein0} \cross \ &= \ q\resoP + q^{-1} \resoN\\
\label{eq.loop0}  \trivloop\  &=\  (-q^2 -q^{-2})\emptyr.
 \end{align}
For details, see Section \ref{sec.skein}.

For an   oriented surface $\Sigma$ let $\SS:= \cS(\Sigma\times (0,1))$. Then $\SS$ has an $\cR$-algebra structure where the product of two framed links $\al_1$ and $\al_2$ is the result of stacking $\al_1$ above $\al_2$.  The skein algebras and modules have played an important role in low-dimensional topology and quantum topology and they serve as a bridge between classical topology and quantum topology. The skein modules have connections to the $\mathrm{SL}_2(\BC)$-character variety \cite{Bullock,PS1}, the quantum group of $\mathrm{SL}_2(\BC)$, the Witten-Reshetikhin-Turaev topological quantum field theory \cite{BHMV}, the quantum Teichm\"uller spaces, see e.g. \cite{CF,Kashaev,BW,Le:QT}, and the quantum cluster algebra theory \cite{Muller}.

\def\SH{\cS(H)} 
 
 \subsection{Main result} When  $\Sigma$ is the boundary $\pM$ of an oriented 3-manifold $M$, there is an action of $\SS$ on $\SM$ given by $\al * \beta= \al \cup \beta$ whenever $\al$ and $\beta$ are framed links, see the details in Section \ref{sec.surf}. This action plays an important role in the theory. For example, various topological quantum field theories are constructed based on this type of actions, see e.g. \cite{BHMV,BW3}. An important question is when this action is faithful. Even when $M$ is a handlebody, the question is important and was open.  
 \begin{thm}\label{thm.1}
 Suppose $H$ is handlebody of genus $g\ge 1$ and $\Sigma$ is  the boundary of $H$. Assume that the ground ring $\cR$ is a commutative domain with unit and  equipped with a distinguished invertible element $q$. 
 
 The action of $\SS$ on $\SH$ is faithful if and only if $q$ is not a root of 1, ie if and only if $q^n\neq 1$ for all positive integers $n$.
 \end{thm}
 
 The result has already found application in the work of J. Cooke and P. Samuelson in their study of double  affine  Hecke  algebras \cite{CS}.
 
 To prove that when $q$ is a root of unity the action of $\SS$ on $\SH$ is not faithful we rely on a result of \cite{BW1,Le:Che} which states that if one threads a framed knot in $H$ by the Chebyshev polynomial $T_N$ of type 2, where $N$ is the order of $q$, then the result is transparent with respect to disjoint union, see Section \ref{sec.proof}. This allows to  construct a non-zero element in $\SS$ which acts trivially on $\SH$.
 
To show that when $q$ is not root of 1 the action of $\SS$ on $\SH$ is faithful we use filtrations of $\SS$ and $\SH$ coming from a pair of pants decomposition of $\Sigma$ and the Dehn-Thurston coordinates of curves on $\Sigma$. The action of the corresponding associated algebra on the associated module is much easier to work with and can be calculated explicitly in ``stable" cases which are enough to prove the faithfulness of the action. R.~Detcherry  informed me that he has an independent proof of the faithfulness for the case when $1-q^n$ is invertible in $\cR$ for all positive integers $n$. 

\subsection{Organization of the paper} We recall the definitions and basic facts of skein modules in Section \ref{sec.skein}. In Section \ref{sec.surf} we  introduce filtrations on skein algebras and recall a result from \cite{FKL}  calculating the product of triangular simple diagrams, introduced in \cite{PS2}, in the associated algebras. Section \ref{sec.proof} contains the proof of the main theorem. 

\subsection{Acknowledgements}The author would like to thank R. Detcherry, P. Samuelson, and R. Santharoubane for helpful discussion. This work is supported in part by NSF grant DMS-1811114.  

\section{Skein modules and algebras} \label{sec.skein}
\def\SF{\cS (F)}
\def\bF{\bar F}
\def\KF{\SF}
\def\cC{\mathcal C}

We recall here the definition of the Kauffman bracket skein module and algebra and their basic properties, including a description of 
the basis in the case of surfaces.

\subsection{Ground ring $\cR$} Throughout the paper the ground ring $\cR$ is a  commutative domain with unity 1, equipped with a distinguished invertible element $q$. The reader might think of $\cR$ as the field $\BC$ of complex number and $q$ is a non-zero complex number, or $\cR=\BZ[q^{\pm 1}]$, where $\BZ$ is the set of integers. We use $\BN$ to denote the set of non-negative integers. Thus our $\BN$ contains 0.
 
\subsection{Skein module} Let $M$ be an oriented 3-manifold, with possibly non-empty boundary.
A {\em framed link}  in  $M$ is an embedding of  a disjoint union of oriented annuli in  $M$. 
By convention the empty set is considered as a framed link with 0 components and is isotopic only to itself.
The orientation of a component of a framed link corresponds to choosing a preferred side to the annulus.

The Kauffman bracket skein module of $M$, denoted by $\SM$, is the $\cR$-module  freely spanned by all isotopy classes of framed links in $M$ subject to the following {\em  relations}
\begin{align}
 \cross \ &= \ q\resoP + q^{-1} \resoN   \label{KBSR}\\
  \trivloop\  &=\  (-q^2 -q^{-2})\emptyr.  \label{KBSR2}
\end{align}
Here the framed links in each expression are identical outside the balls pictured in the diagrams. The skein relations were introduced by Kauffman \cite{Kauffman} where he showed that when $M=\BR^3$, the three-space, one has $\SM= \cR$, and the value of a framed link is in $\cR$ is the framed version of the Jones polynomial \cite{Jones}.  Kauffman bracket skein modules for oriented 3-manifolds were introduced independently by Turaev and Przytycki \cite{Przy,Turaev} in an attempt to generalize the Jones polynomial to general 3-manifolds. Turaev \cite{Turaev2} used the algebra structure, see below, to quantize the Goldman bracket of the surface.

\def\sS{\mathscr S}

\subsection{Skein algebra of surfaces} Suppose $F$ is an oriented surface, with possibly non-empty boundary. Define $\SF = \cS(F\times (0,1))$. The $\cR$-module $\SF$ has an algebra structure coming from stacking.  More precisely, the product of two links is defined by placing the first link above the second in the direction given by the interval $(0,1)$. The empty link is the unit.

If $D$ is a non-oriented link diagram on $F$ then a regular neighborhood of $D$ is a framed link, defined uniquely up to isotopy. We identify $D$ with the isotopy class of the framed link it defines.  Any framed link in $F \times (0,1)$ is isotopic to a framed link determined by a link diagram. A {\em simple diagram} on  $F$ is a link diagram with no crossings and no {\em trivial loops}, i.e. a curve bounding a disk.  We consider the empty set as a simple diagram which is isotopic only to itself.

Let $B(F)$ be the set of all isotopy classes of simple diagrams on $F$. It is known 
 \cite{PS1} that $\SF$ is a free $\cR$-module with basis $B(F)$.
Every  non-zero $ \al \in \KF$ has a unique {\em standard presentation}, which is a finite sum
\be
\label{eq.pres} 
\al = \sum_{j\in J} c_j \al_j,
\ee
where $c_j\neq 0$ and $\al_j\in B(F)$ are distinct. 

\def\kinkp{  \raisebox{-8pt}{\incl{.8 cm}{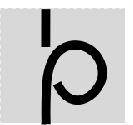}} }
\def\kinkn{  \raisebox{-8pt}{\incl{.8 cm}{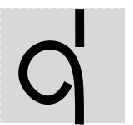}} }
\def\kinkzero{  \raisebox{-8pt}{\incl{.8 cm}{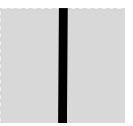}} }

If a crossing of a link diagram $D$ is as in the left hand side of \eqref{KBSR}, then 
the first (resp. second) diagram on the right hand side will be referred to as the $+1$ (resp. $-1$) {\em smooth resolution} of the left hand side diagram at the crossing.

Suppose $D$ is a link diagram on $F$. Let $\cC$ be the set of all crossings. For every map $\sigma: \cC \to \{\pm1 \}$ let $D_\sigma$ be the non-crossing diagram obtained from $D$ by doing the  $\sigma(c)$ smooth resolution at every crossing $c$.
Let $|\sigma|= \sum_{c\in \cC}  \sigma(c)$ and $l(\sigma)$ be the number of trivial components of $D_\sigma$. Let $D_\sigma'$ be the simple diagram obtained from $D_\sigma$ by removing all the trivial components. Using the two  skein relations, we have 
\be 
\label{eq.prod}
D = \sum_{\sigma: \cC \to \{\pm 1\}} q^{|\sigma|} (-q^2 - q^{-2})^{l(\sigma)} \, D_\sigma' \quad \text{in } \KF.
\ee
 
From the defining  relations, one can easily derive the following relation:
\begin{align}
 -q^{-3} \kinkp \ &= \ \kinkzero\  = \ -q^3\kinkn \label{eq.kink}
\end{align}

\section{Curves on surfaces and coordinates}\label{sec.surf}

\def\pF{\partial F}
\def\oF{\mathring F}

We recall how to coordinatize the system of simple diagrams on surfaces in both cases when the surface has boundary and when the surface is closed. We introduce filtrations on the skein algebra of a surface using a pair of pants decomposition, and explain how the product in the associated algebra look like for a class of so-called triangular curves.

\subsection{Edge coordinates and Thurston form on compact non-closed surfaces}
\label{sec.edge}

Suppose $F$ is a compact connected oriented surface with non-empty boundary $\pF$ and  $\chi(F)<0$, where $\chi(F)$ is the Euler characteristic. 
Then $\oF=F \setminus \pF$ is known as a {\em hyperbolic punctured surface}.  We will now  introduce  the notion of {\em $\pF$-triangulation} of $F$, which is essentially the ideal triangulation of $\oF$.

A {\em $\pF$-arc} is a proper embedding $[0,1] \embed F$. {\em Isotopies of $\pF$-arcs} are considered in the class of $\pF$-arcs. A $\pF$-arc is {\em trivial} if it can be isotoped into a small neighborhood of the boundary $\pF$. A {\em $\pF$-triangulation} $\cE$ of $F$ is a maximal collection of pairwise disjoint and pairwise non-isotopic non-trivial $\pF$-arcs. An element $a\in \cE$ is called an {\em edge} of the triangulation. 
It is easy to show that $|\cE|=- 3 \chi(F)$. A connected component of $F \setminus \left( \bigcup_{e\in \cE} e  \right)$ is called a {\em face} of the triangulation. An edge $e\in \cE$ is a {\em side } of a face $\tau$ if the topological closure of $\tau$ contains $e$. An edge is {\em self-folded } if it is the side of only one face.
There are always triangulations of $F$ without self-folded edges. For simplicity we will assume that the triangulation $\cE$ does not have any self-folded edge, although all the facts in this paper can be easily adapted to the case when the triangulation has self-folded edges.

For an edge $e\in \cE$ and a simple diagram $\al$ on $F$ the {\em geometric intersection number} $I(e, \al)$ is the minimum of the  cardinalities of $\al'\cap e$, where $\al'$ run the set of simple diagrams isotopic to $\al$. 

Three distinct edges $a,b,c\in \cE$ are {\em triangular} if they are sides of a face. Recall that $B(F)$ is the set of isotopy classes of simple diagrams on $F$. We use the notation $X^Y$ for the set of all functions from $Y$ to $X$. 
Define the map
$$ \nu: B(F) \to \BN^\cE \quad \text{
by} \   
 \nu(\al)(e) = I(\al, e).
$$
 It is known that $\nu$ is injective, and its image is the submonoid $\Lambda=\Lambda(F,\cE)\subset \BN^\cE$ consisting of functions $f: \cE \to \BN$ such that
whenever $a,b,c$ are triangular, we have
 $f(a)+f(b)+f(c)$ is even and
 $f(a)\le f(b)+f(c)$.  For details, see e.g.~\cite{Matveev}. We will denote 
 $$ S: \Lambda(F,\cE) \to B(F)$$
 the inverse function of $\nu$.

\def\bm{{\mathbf m}}
Associated to the triangulation $\cE$ is an anti-symmetric function
$$ Q : \cE \times \cE \to \BZ,$$
known as the Thurston form, which determines the Poisson structure on the Teichm\"uller space of $\oF$. For details see e.g. \cite{BW,FoG,FST,Penner}. Although the exact definition is not important for us, let us recall  
the definition of $Q$. Each face $\tau$ of the triangulation is planar, hence there is defind a cyclic order of the 3 sides of $\tau$, using the counter clockwise orientation of the boundary of $\tau$. For two edges $a,b\in \cE$ and a face $\tau$ define $Q_\tau(a,b)=0$ if one of $a,b$ is not a side of $\tau$, and otherwise define
$Q_\tau(a,b)=1$ or $-1$ according as $b$ follows $a$ along the cyclic order or not. Then
$$ Q(a,b) = \sum_{\tau:\  \text{ faces }} Q_\tau(a,b).$$
The bilinear form $\la \cdot, \cdot \ra_Q: \BZ^\cE \times \BZ^\cE \to \BZ$ defined by
\be 
\la \bn , \bm \ra_Q := \sum_{a,b\in \cE} Q(a,b) \bn(a) \bm(b) 
\label{eq.Q}
\ee
plays an important role in the theory of Teichm\"uller spaces of $\oF$.

\def\da{\boldsymbol{\delta}_a}
\def\tte{{\tilde e}}
\def\cP{{\tilde {\cE}}}

\subsection{Dehn-Thurston coordinates for curves on closed surfaces}\label{sec.DT} 
Suppose $\Sigma$ is a closed oriented surface of genus $g>1$. Then $\Sigma$ is the double of a compact surface $F$ with $\chi(F)=1-g$ with non-empty boundary $\pF$. This means $\Sigma= F \sqcup F'/\sim$, where $F'$ is a copy of $F$ and $x\sim x'$ for $x\neq x'$  if and only if $x\in \pF$ and $x'$ is the corresponding point of $x$ in $\partial F'$. By a {\em Dehn-Thurston datum} of $\Sigma$ we mean a pair $(F, \cE)$, where  $\Sigma$  is the double of $F$ and $\cE$ is a $\pF$-triangulation of $F$.

For each $e\in \cE$ let $e'\subset F'$ be the corresponding $\pF'$-arc of $F'$. 
Then $\tte = e\cup e'$ is a simple closed curve in $\Sigma$. The collection $\cP= \{ \tte \mid e\in \cE\}$ is a {\em pair of pants decomposition} of $\Sigma$.

Recall that $B(\Sigma)$ is the set of isotopy classes of simple diagrams on $\Sigma$. The Dehn-Thurston datum $(F, \cE)$ gives rise to the Dehn-Thurston coordinate map 
$$ \mu: B(\Sigma) \xrightarrow{\cong  } \Gamma \subset \BZ^{\cE \cup \cE^*},$$
where $\cE^*=\{ e^* \mid e \in \cE\}$ is a copy of $\cE$.  The set $\Gamma$ is a submonoid of $\BZ^{\cE \cap \cE^*}$.  By definition $\mu(\al)(e)= I(\al, \tte)$  for $e\in \cE$. The definition of $\mu(\al)(e^*)$ for $e^*\in \cE^*$, called the {\em twist coordinate corresponding to} $\tte$, is more involved, and since we we don't need the exact definition  we refer the reader to \cite{LS} for details. We will recall here a couple of properties of the Dehn-Thurston coordinates that we will need later.

Firstly suppose $\al\subset F$. One can consider $\al$ is a simple diagram on $F$ and  define $\nu(\al):\cE\to \BN$ as in Subsection \ref{sec.edge}. Then the restriction of $\mu(\al)$ onto $\cE$ is equal to $\nu(\al)$ while the restriction of $\mu(\al)$ onto $\cE^*$ is the 0 function. In other words, 
 $\mu(\al)(e)= \nu(\al)(e)$ for all $e\in \cE$ and $\mu(\al)(e^*)=0$ for all $e^*\in \cE^*$, i.e. all the twist coordinates of $\al$ are 0.

Secondly  suppose $\al$ is a simple diagram on $\Sigma$ and $e\in \cE$ such that $I(\al,\tte) =n>0$. 
One defines a simple digram $T_e(\al)$ as follows. A closed tubular neighborhood $A(\tte)$ of the simple closed curve $\tte$ is homeomorphic to an annulus $S^1 \times [0,1]$ where $\tte$ is identified with $S^1 \times \{1/2\}$. A segment of the form $z \times [0,1]$, where $z\in S^1$, is called a vertical line in $A(\tte)$.
 We assume that $\al$ intersects $A(\tte)$ in $n$ vertical lines. 
Then $T_e(\al)$ is the the result of twisting the $n$ vertical lines along 
 $\tte$ by $2\pi/n$, see Figure \ref{fig:Twist}. 
The twist coordinates are defined so that 
\be \mu(T_e(\al))(e^*) = \mu(\al)(e^*) +1. \label{eq.Twist}
\ee

\FIGc{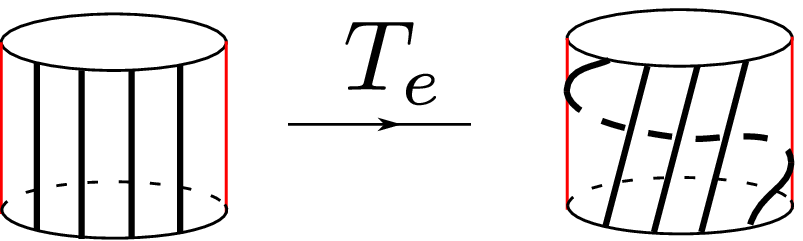}{The map $T_e$. Here $n=4$.The cylinder is a tubular neighborhood $A(\tte)$ of $\tte$. 
The red lines are part of $\pF$.}{2cm}

 Let $S: \Gamma \to B(\Sigma)$ be the inverse of $\mu$.

Following \cite{PS2}, a simple diagram $\al$ is called {\em $\cE$-triangular} if it can be obtained from a simple diagram $\beta$ in $F$ by repeatedly applying the map $T_e$ or $(T_e)^{-1}$, with possibly many $e\in \cE$. In other words, $\al$ is triangular if 
\begin{itemize}
\item whenever $a,b,c\in  \cE$ are triangular one has $\mu(\al) (a) \le \mu(\al)(b) + \mu(\al)(c)$, and
\item whenever  $\mu(\al)(e)=0$  one has $\mu(\al)(e^*)=0$.
\end{itemize}
 
We will see later that the product of triangular diagrams behaves well in the filtration defined by the system $\{\tte\}$ .

\subsection{Torus} \label{sec.Torus}
Suppose $\Sigma$ is the torus, i.e. $g=1$. The skein algebra $\SS$ in this case is well understood; it is a subalgebra of a quantum torus \cite{FrG} and many of its properties can be deduced from this fact.

Note that $\Sigma$ is the double of an annulus $F= S^1 \times [0,1]$. We excluded this case in the previous discussion since technically there is no triangulation of $F$. However, in order to uniformly treat all closed surfaces, we use the following convention. A {\em $\pF$-triangulation} $\cE=\{e\}$ of $F$ consists of one $\pF$-arc $e$ of the form $\{x\} \times [0,1]$ where $x\in S^1$. The simple closed curve $\tte$ will be called a {\em meridian} of $\Sigma$, while the curve $S^1 \times \{1/2\} $ on the annulus $F$ will be called a {\em longitude} of $\Sigma$. Again a Dehn-Thurston datum of $\Sigma$ is the pair $(F, \cE)$.

The set $B(F)$ of isotopy classes of simple diagrams on $F$ is parameterized by $\Lambda=\BN$, where $S(m)$ is the diagram consists of $m$ parallel copies of the longitudes. The anti-symmetric form $Q: \cE \times \cE\to \BZ$ is defined by $Q(e,e)=0$.

The twist $T_e$ is defined exactly as in the previous case.
The set $B(\Sigma)$ is parameterized by 
$$\Gamma=\{ (n,t)\in \BN\times \BZ \mid  \text{if} \ n=0 \ \text{then} \ t\ge 0\}.$$
Here $S(n,t)= (T_e)^{t}(S(n))$ if $n >0$ and $S(0,t)$ is the disjoint union of $t$ copies of the meridian $\tte$. The simple diagram $S(n,t)$ is defined to be {\em triangular} if $n>0$ or $n=0$ and $t=0$.

\def\Gr{\mathrm{Gr}   }
\def\SF{{\cS(F)}}
\def\SH{\cS(H)}
\def\bm{{\mathbf{m}}}
\def\pH{\partial H}

\subsection{Action of $\SS$ on $\SH$} \label{sec.action}
Suppose $\Sigma$ is the boundary of an compact oriented 3-manifold $M$.
We identify a closed  tubular neighborhood $\cN(\pM)$ of $\pM$ with $\Sigma\times [0,1]$ so that $\pM = \Sigma\times \{0\}$. Every framed link  in $M$ can be isotoped off $\cN(\pM)$.
If $\al$ is a framed link in $\Sigma\times (0,1)$ and $\beta$ is a framed link in $M \setminus \cN(\pM)$ then define $\al * \beta := \al\cup \beta$ as an element of $\SM$. This defines an action of $\SS$ on $\SM$.

Suppose now $F$ is a compact oriented connected surface with non-empty boundary $\pF$ and with Euler characteristic $\chi(F) =1-g$. Then $H= F \times [0,1]$ is a handlebody of genus $g$. By definition, we identify $\SH$ with $\SF$. 

Note that $\Sigma:=\pH$ is the double of $F$.  The embedding $\iota: F \embed \Sigma$ gives rise to an algebra embedding $\iota_*: \SF \embed \SS$. 
For $\al\in \SS$ and $\beta \in \SF=\SH$, from 
the definition one has
  
\be \al * \beta  = [\al \,  \iota_*(\beta)]  * 1. 
\label{eq.action}
  \ee
  
  \def\bal{\bar \al}
  \def\hO{\bar \Omega}
  \def\bB{\bar B}

  \subsection{Filtrations}\label{sec.fil}
  
   Return to a closed surface of genus $g\ge 1$, equipped with a Dehn-Thurston datum $(F,\cE)$.

  For $k\in \BN$ let $F_k(\SS)\subset \SS$ be the $\cR$-submodule spanned by $\al\in B(\Sigma)$ such that 
  $$\sum_{e\in \cE} I(\al, \tte) \le k.$$
  Then $\{ F_k(\SS) \}_{k=0}^\infty $ is an exhaustive increasing filtration of $\SS$, compatible with the product structure. This type of filtrations of skein aglebras and modules are used in many works, for example \cite{Le:Che,Marche,FKL,PS2}.
  
 We consider $\SS$ as an $\BN$-filtered algebra with the filtration $\{ F_k(\SS) \}_{k=0}^\infty $. Let $\Gr \SS$ denote the $\BN$-graded associated algebra. By definition
  $$ \Gr \SS = \bigoplus_{k=0}^\infty F_k(\SS) / F_{k-1}(\SS), 
  $$
  where by convention $F_{-1}(\SS)$ is the 0 module. Suppose $x\in \SS$ is non-zero. Then there is $k\in \BN$, called the {\em degree of $x$}, such that $x\in F_k(\SS) \setminus F_{k-1}(\SS)$.
   Let 
   $$\bar x\in  F_k(\SS) / F_{k-1}(\SS) \subset \Gr \SS$$ be the image of $x$ under the natural projection $  F_k(\SS)  \to F_k(\SS) / F_{k-1}(\SS)$.

Let $B_k(\Sigma) \subset B(\Sigma)$ be the subset consisting of isotopy classes of simple diagrams $\al$ with $\sum_{e\in \cE} I(\al, e) \le k$. Then 
 $F_k(\SS)$ is free $\cR$-module with basis $B_k(\SS)$. It follows that the set
 $$ \bB(\Sigma): =\{ \bal \mid \al \in B(\Sigma)\}$$
 is a free basis of the $\cR$-module $\Gr\SS$.

  Similarly let $F_k(\SF)\subset \SF$ be the $\cR$-submodule spanned by $\al\in B(F)$ such that 
  $$\sum_{e\in \cE} I(\al, e) \le k.$$
  Then $\{ F_k(\SF) \}_{k=0}^\infty $ is an exhaustive increasing filtration of $\SF$. Let $\Gr\SF$ be the corresponding $\BN$-graded $\cR$-module. Again for a non-zero element $x\in \SF$ one can define $\bar x \in \Gr \SF$.The set
  $$ \bB(F): =\{ \bal \mid \al \in B(F)\}$$
 is a free basis of the $\cR$-module $\Gr\SF$.
  
  The $\BN$-filtration $\{ F_k(\SF) \}_{k=0}^\infty $  of $\SF$ is compatible with the $\BN$-filtration $\{ F_k(\SS) \}_{k=0}^\infty $ under the action of $\SS$ on $\SF$. In other words, 
if $\al\in F_k(\SS)$ and $\beta \in F_m(\SF)$, then (from the definition) one has $\al * \beta \in F_{k+m}(\SF)$. It follows that we have an associated action of $\Gr\SS$ on $\Gr\SF$.

Recall that we have defined triangular simple diagrams on a closed surface $\Sigma$ equipped with a Dehn-Thurston datum $(F,\cE)$, see Subsections \ref{sec.DT} and \ref{sec.Torus}. Suppose $\bn\in \BZ^\cE$ and $\bt\in \BZ^{\cE^*}$, then $(\bn,\bt)\in \BZ^{\cE \cup \cE^*}$, and we say that $(\bn,\bt)$ is {\em triangular} if $(\bn,\bt) \in \Gamma$ and  the simple diagram $C(\bn,\bt)$, which has Dehn-Thurston coordinates $(\bn,\bt)$, is triangular.
We say an element $x\in \Gr\SS$ is {\em triangular} if it is an $\cR$-linear combination of $\overline{C(\bn_i, \bt_i)  }$, where each $(\bn_i, \bt_i)\in \Gamma$ is  triangular.

Let $\Omega=\pF$, considered as an element of $\SS$. As $\Omega\subset F$, it is a triangular simple diagram.
We have the following result of \cite{PS2}.
\bpro  \label{r.PS2}
Suppose $\Sigma=\Sigma_g$ is a surface of genus $g$, equipped with Dehn-Thurston datum $(F, \cE)$. 

(a) The associated algebra $\Gr \SS$ is a domain.

(b) Suppose $x\in \Gr \SS$. For sufficiently large integers $k$ the elements $\hO^k x $ are triangular.

\epro

The reason why we consider triangular simple diagrams is that the product of two triangular simple diagrams in the associated algebra is very simple, as given in the following proposition, a refinement of a result of \cite{PS2}.

\bpro [Proposition 4.9 of \cite{FKL}]
Suppose $(\bn,\bt), (\bn', \bt')\in \Gamma$ are  triangular. Then 
\be 
\overline {C(\bn, \bt)} \, \, \overline {C(\bn', \bt')} = q^{\la \bn, \bn'\ra_Q + \bt\cdot \bn' - \bt' \cdot \bn }  \overline {C(\bn+\bn', \bt+ \bt')}.
\label{eq.1b}
\ee
\epro
Here $\la \bn, \bn'\ra_Q$ is a  bilinear form defined in Subsection \ref{sec.edge}, and $\bt\cdot \bm$, for $\bt\in \BZ^{\cE^*}$ and $\bm\in \BZ^\cE$, is the dot product defined by
$$ \bt\cdot \bm = \sum_{e\in \cE} \bt(e^*) \bm(e).$$

 \section{Proof of Theorem \ref{thm.1}} \label{sec.proof}
 \def\fr{{\mathrm{fr}}}
 
 Here we prove Theorem \ref{thm.1}, which for convenience of the reader is reformulated here.
 \begin{theorem}\label{thm.1a}
 	Suppose $M$ is handlebody of genus $g\ge 1$ and $\Sigma$ is  the boundary of $M$. Assume that the ground ring $\cR$ is a commutative domain with unit and a distinguished invertible element $q$. 
 	
 	The action of $\SS$ on $\SM$ is faithful if and only if $q$ is not a root of 1, ie if and only if $q^n\neq 1$ for all positive integers $n$.
 \end{theorem}
 
 \subsection{Roots of 1} 
Define the Chebyshev polynomial $T_n(z) \in \BZ[z]$ by the recursion
$$ T_0(z) =2, T_1(z) = z, T_n (z) =z  T_{n-1}(z) - T_{n-2}(z) \ \text{for } \ n \ge 2.$$
The polynomial $T_N(z)$ is characterized by $T_N(u+ u^{-1})= u^N + u^{-N}$.

For a framed knot $\al$ in $M$ and a positive integer $k$ let the $k$-th framed power $\al^{(k)}$ be the disjoint  union of $k$ parallel push-off copies of $\al$, where the push-off is in the direction of the framing of $\al$. If $k=0$ then $\al^{(k)}$ is the empty link by convention. For a polynomial $P(z)= \sum c_i z^i$ and a framed knot $\al$ let $P^\fr(\al) \in \SM$ be 
$$ P^\fr(\al) =\sum c_i \al^{(i)},  $$
considered as an element of $\SM$. Moreover, suppose $\beta$ is a framed link disjoint from $\al$.
Presenting each $\al^{(k)}$  as  framed link in a small neighborhood of $\al$, let 
$$ P^\fr(\al) \cup \beta := \sum c_i [ \al^{(i)} \cup \beta], $$
which is well-defined as an element of $\SM$, and depends only on the isotopy class of $\al\cup \beta$.

We have the following tranparency property of 
the element $(T_{N})^\fr(\al)$ when  $q^N=1$.

\bpro[Corollary 2.7 of \cite{Le:Che}] \label{r.trans}
Suppose $q^N=1$ for a positive integer $N$. Assume $\al$ and $\al$ are two
isotopic framed knots in $M$, and $\beta$ is a framed link disjoint from both $\al$ and $\al'$.
Then 
\be  (T_N)^\fr(\al) \cup \beta  = (T_N)^\fr(\al') \cup \beta.
\label{eq.trans}
\ee
\epro 
 
In picture, Identity \eqref{eq.trans} is depicted in Figure \ref{fig:Trans}.

\FIGc{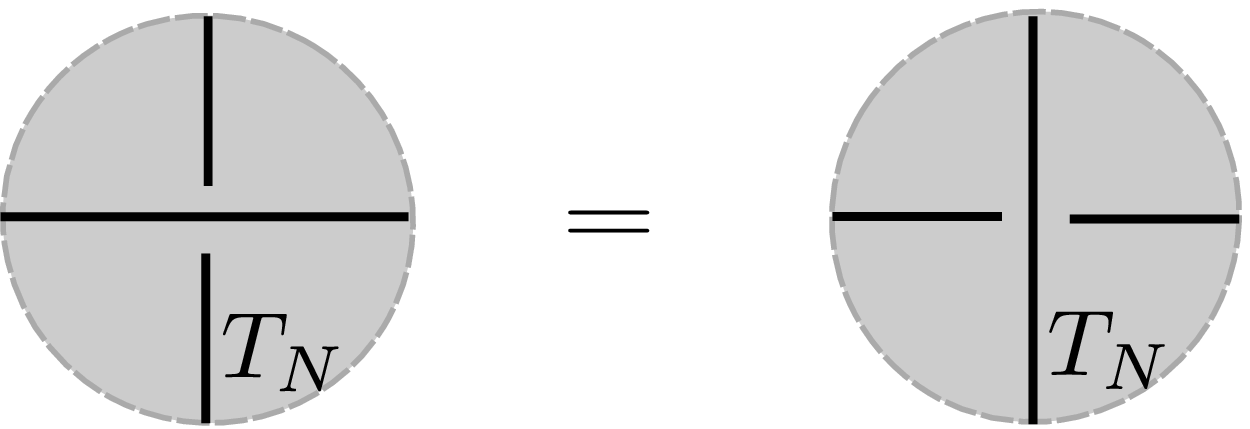}{Transparency of $(T_N)^\fr(\al)$. Here  the vertical line is part of $\al$ to which $T_n$ is applied,  and the horizontal line is part of another framed link.}{1.8cm}

The transparency property in the 3-manifold case, Proposition \ref{r.trans},  is a generalization of similar property for links in the thickened surfaces \cite{BW1}. For further generalizations and applications see \cite{LP,BL}.
 \begin{proof}[Proof of Theorem \ref{thm.1} for the case $q$ is a root of 1] Suppose $q^N=1$ for a positive integer $N$. Choose a simple loop $\al$ on $\Sigma$ which is non-trivial in $\Sigma$ but bounds a disk in $H$.
 Since $\al$ is isotopic to a trivial loop $\gamma$ in $H$, the transparency property of Proposition \ref{r.trans} shows that for all $a\in \SH$,
 \be  [(T_N)^\fr(\al) - (T_N)^\fr(\gamma) ] * x =0 \label{eq.tra}
 \ee
 As $\gamma$ is trivial, we have $\gamma = -q^2 -q^{-2}$ in $\SH$, and hence
  $$T_N(\gamma)= (-q^2)^N + (-q^{-2})^N= 2(-1)^N.$$ 
Since the elements $\al^{(k)}$ for $k=0,1,2,\dots$ are distinct elements of the basis $B(\Sigma)$, the element 
$a:=(T_N)^\fr(\al)- 2(-1)^N $ is non-zero in $\SS$. Identity \eqref{eq.tra} shows that $a$ acts trivially on $\SH$. This shows the action of $\SS$ on $SH$ is not faithful.
 \end{proof}
 
 The remaining part is devoted to a proof for the case when $q$ is not a root of 1.

 \subsection{A lemma on root of 1} The following lemma is  where we use the fact that $q$ is not a root of 1.
 \blem \label{r.VD} Assume that  $r, k$ be  positive integers, and $q\in \cR$ is not a root of 1.  Suppose that
 $c_i\in \cR$ and $\bn_i\in \BZ^r$ for $i=1,2,\dots, k$, where the $\bn_i$ are pairwise distinct.
Let $\Lambda\subset \BZ^r$ be a submonoid generating a subgroup of rank $r$ in $\BZ^r$. If \be \sum_{i=1}^k c_i q^{\bn_i \cdot \bm}=0
\label{eq.kk}
\ee for all $\bm \in \Lambda$ then $c_i=0$ for all $i=1,\dots, k$.

\elem
\begin{proof} 
For a non-zero element $x$ in the standard Euclidean space $\BR^r$ let 
$$x^\perp=\{ y\in \BR^r \mid x \cdot y=0\},$$
 which is a subspace of $\BR^r$ of dimension $r-1$. The finite collection $(\bn_i- \bn_j)^\perp, i\neq j$, cannot cover the whole $\BR^r$. The condition on the rank of $\Lambda$ shows that there is $\bm \in \Lambda$ not belonging to any of $(\bn_i-\bm_j)^\perp$. This means the numbers $s_i := \bn_i  \cdot \bm$ are pairwise distinct.

Replace $\bm $ by $j \bm$ for $j=0,1,\dots, k-1$ in \eqref{eq.kk}, we get
$$ \sum_{i=1}^k c_i (s_i)^j =0, \qquad j=0,1,\dots, k-1.$$
One can solve for $c_i$ from the above linear system. The determinant is the 
Vandermonde determinant, which is non-zero since the $s_i$ are disctinct. It follows that all $c_i$ are 0.
\end{proof}
 
 \subsection{Main technical result} Let $\Sigma$ be a closed surface of genus $g\ge 1$. We assume $\Sigma=\pH$, where $H= F \times [0,1]$, as in Subsection \ref{sec.action}. Choose a $\pF$-triangulation $\cE$ of $F$. Thus  $(F, \cE)$ is 
   a Dehn-Thurston datum $(F, \cE)$ of $\Sigma$. The set $B(\Sigma)$ of isotopy classes of simple diagrams  on $\Sigma$ is parameterized by the submonoid $\Gamma \subset \BZ^{\cE \cup \cE^*}$. For $(\bn,\bt)\in \Gamma$, where $\bn\in \BZ^\cE$ and $\bt\in \BZ^{\cE^*}$,  recall that $C(\bn, \bt)$ is the simple diagram with coordinates $(\bn, \bt)$. The set $B(F)$ of simple diagram on $F$ is parameterized by the submonoid $\Lambda\subset \BN^\cE$. For $\bm \in \BN^\cE$ recall that $S(\bm)$ is the diagram on $F$ with coordinates $\bm$. Also we identify $\SH$ with $\SF$.
 
As explained in Subsection \ref{sec.fil} we have $\BN$-filtrations on $\SS$ and $\SF$, and the corresponding associated graded algebra $\Gr \SS$ and associated graded   module $\Gr \SF$. 
  For a non-zero element $x$ we defined $\bar x$ in the associated modules.

The main technical result is the following, which calculate the geometric action of the associated algebra $\Gr\SS$ on the associated module $\Gr\SF$.
 
 \bpro \label{r.r1}
  Suppose $(\bn,\bt)\in \Gamma$ is triangular and $\bm \in \Lambda$. One has
 \be 
\overline {C(\bn, \bt)} *  \overline { S(\bm)} = u(\bn,\bt)  q^{2 \bt\cdot \bm}  q ^{\la  \bn, \bm \ra_Q } \overline { S(\bn+\bm)}.
\label{eq.main}
\ee
where $u(\bn,\bt)= (-q^2)^{\la \bt \ra}  q^{ \bt \cdot  \bn }  $ is a unit in $R$. Here $\la \bt \ra = \sum_{e\in \cE} \bt(e^*)$ and $\la  \bn, \bm \ra_Q$ is defined by Equation~\eqref{eq.Q}.
 
 \epro

 First we prove a lemma, which deals with modifications of  twist coordinates.
 \blem \label{lem.1a}
 Suppose $(\bn,\bt)\in \Gamma$ is triangular. 
 Fix  $e\in \cE$ and modify $\bt$ to $\bt'$ so that $\bt'(e^*)=0$ while $\bt'(c^*) = \bt(c^*)$ for $c\neq e$. Then
\be 
\overline {C(\bn, \bt)} *  1  = (-q^{\bn(e)+2})^{\bt(e)}\, \,  \overline {C(\bn, \bt')} *  1 .
\label{eq.main2}
\ee
\elem
\begin{proof} If $\bt(e)=0$ then $\bt=\bt'$ and the lemma is trivial. 
First assume $\bt(e) > 0$. Present the annulus $A(\tte)$, a closed tubular neighborhood of the simple closed curve $\tte$, as $\tte\times [0,\bt(e)]$ such that for each  $k=0,1,\dots, \bt(e)-1$, the intersection $C(\bn,\bt)\cap   (\tte \times [k,k+1])$ is presented by diagram $x$ depicted in Figure \ref{fig:Twist3}, where we also describe diagrams $y$ and $z$ on $F$.

\FIGc{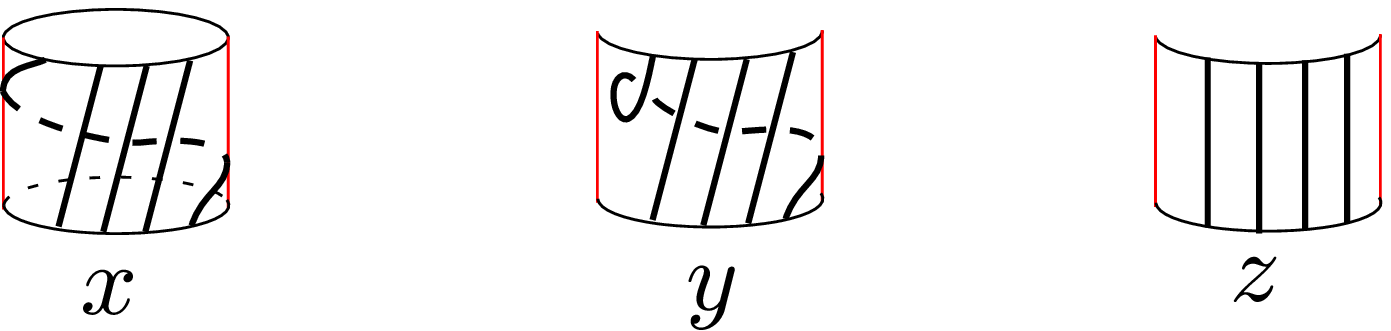}{Diagram $x$ on $\Sigma$ and diagrams $y,z$ on $F$. The vertical red lines are boundary of $F$.}{2cm}

 Considering $x$ as an element in $H$,  we see that $x$ is isotopic to the diagram $y$ on $F$. Note that $y$ has one positive self-crossing and $n-1$ other crossings. Removing the self-crossing results in a factor $-q^3$, see Equation \eqref{eq.kink}.
  For each of the remaining $n-1$ crossing, there is only one smooth resolution, namely the $+1$ resolution, which does not decrease the geometric intersection with $\tte$. As a result, when applying Formula \eqref{eq.prod} to diagram $y$,  modulo elements whose geometric intersections with $e$ are less than $\bt(e)$, we have
$$y= (-q^3) q^{n-1} z= - q^{n+2} z,$$
where $z$ is the diagram presented by $n$ vertical lines, see Figure \ref{fig:Twist3}. Doing this to each $k=0,1,\dots, t-1$,  we get 
Identity \eqref{eq.main2}. 

The case $\bt(e) <0$ is similar, where all the crossings now should have $-1$ resolution. \end{proof}

By applying Lemma \ref{lem.1a} to all $e\in \cE$ we get
 \begin{align}
 \overline {C(\bn, \bt)} * 1  & = (-q^2)^{\la \bt \ra}  q^{\bt\cdot \bn}\,  \overline {C(\bn, \mathbf 0 )} *  1
 \notag  \\
& =  
(-q^2)^{\la \bt \ra}  q^{\bt\cdot \bn}\,   \,  \overline {S(\bn)}.  \label{eq.m5}
\end{align}

 \begin{proof}[Proof of Proposition \ref{r.r1}]
 
Since $\overline { S(\bm)} =   \overline { C(\bm, \mathbf 0)} *1 $, we have

\begin{align*}
\overline {C(\bn, \bt)} *  \overline { S(\bm)} &  =  \overline {C(\bn, \bt)} \, \,   \overline { C(\bm, \mathbf 0)} *1 \\
&  = 
q^{\la \bn, \bm\ra_Q + \bt\cdot \bn }  \,  \overline {C(\bn+\bm, \bt)} * 1  \qquad \text{ by \eqref{eq.1b}   }
\\
&= q^{\la \bn, \bm\ra_Q + \bt\cdot \bn }  \, (-q^2)^{\la \bt \ra} q ^{ \bt \cdot \bn + \bt \cdot \bm}\,  \overline { S(\bn + \bm)}  \qquad \text{ by \eqref{eq.m5}   },
\end{align*} 
which is Identity \eqref{eq.main} that we need to prove.
\end{proof}

\subsection{Proof of Theorem \ref{thm.1}, the case $q$ is not a root of 1}
\begin{proof} If the associated action of $\Gr \SS$ on $\Gr \SH$ is faithful, then so is the action of $\SS$ on $\SH$. Hence it is enough to show that the associated action is faithful.

Assume that a non-zero $x\in \Gr\SS$ acts trivially on $\Gr \SH$. By Proposition \ref{r.PS2}, there is a positive integer $k$ such that $y:=\Omega^k x$ is triangular. Since $\Gr \SS$ is a domain by Proposition \ref{r.PS2}, we have $y\neq 0$. Clearly $y * \Gr \SF=0$.

Since $y$ is triangular, we have the standard presentation  $y = \sum_{j\in J} c_j \overline{C(\bn_j, \bt_j)}$, where $J$ is a non-empty set, the $(\bn_j, \bt_j)$ are triangular and  pairwise distinct, and each $c_j$ is a non-zero element of $\cR$. For  $\bm \in \Lambda$
we have
\begin{align}
0 & = y *\overline{ S(\bm)}  =   \sum_{j\in J} c_j   \, \overline{C(\bn_j, \bt_j)}  *  \overline{S(\bm)} \notag  \\
&=  \sum_{j\in J} c_j' q^{2 \bt_j \cdot \bm} q^{\la \bn_j, \bm\ra_Q} \overline{S(\bn_j +\bm})
\label{eq.8a}
\end{align}
where we have used Identity \eqref{eq.main} to get Identity \eqref{eq.8a}, and $c'_j:= c_j u(\bn_j, \bt_j)$. Since $u(\bn_j, \bt_j)$ is a unit, we have that $c'_j\neq 0$ for all $j\in J$.

Choose $j_0\in J$ and let $J'= \{ j \in J \mid \bn_j = \bn_{j_0}\}$. As the
set $\{ \overline{S(\bk)} \mid  \bk \in \Lambda\}$ is a free basis of $\Gr \SF$, the coefficient of $\overline{S(\bn_{j_0} +\bm})$ in the right hand side of Equation \eqref{eq.8a} must be 0. Hence we have 
$$0=  \sum_{j\in J'} c_j' q^{2 \bt_j \cdot \bm} q^{\la \bn_{j_0}, \bm\ra_Q} = q^{\la \bn_{j_0}, \bm\ra_Q} \,  \sum_{j\in J'} c_j' q^{2 \bt_j \cdot \bm}. $$
It follows that for all $\bm \in \Lambda$ one has $\sum_{j\in J'} c_j' q^{2 \bt_j \cdot \bm}=0$.
Lemma \ref{r.VD} shows that all $c'_j$ are 0, a contradiction. This completes the proof of Theorem \ref{thm.1}.
\end{proof}


\end{document}